\documentclass[11pt]{amsart}

\usepackage{amsfonts,amsmath,latexsym,amssymb,verbatim,amsbsy,amsthm,mathtools}
\usepackage{BOONDOX-cal}
\usepackage{graphicx}
\usepackage{soul}
\usepackage{dsfont}
\usepackage{bm}

\usepackage[top=1in, bottom=1in, left=1in, right=1in]{geometry}

\usepackage[dvipsnames]{xcolor}

\usepackage[pdfusetitle,colorlinks=true, pdfstartview=FitV, linkcolor=RoyalBlue,citecolor=ForestGreen, urlcolor=blue]{hyperref}

\usepackage{changes}

\definecolor{labelkey}{rgb}{0,0,1}

\theoremstyle{plain}
\newtheorem{THEOREM}{Theorem}[section]

\newtheorem{theorem}[THEOREM]{Theorem}

\newtheorem{lemma}[THEOREM]{Lemma}

\theoremstyle{definition}

\theoremstyle{remark}

\newtheorem{remark}[THEOREM]{Remark}

\newtheorem{example}[THEOREM]{Example}

\newcommand{\thm}[1]{Theorem~\ref{#1}}
\newcommand{\lem}[1]{Lemma~\ref{#1}}

\newcommand{\rem}[1]{Remark~\ref{#1}}

\DeclareMathOperator{\Id}{Id} %

\def \a {\alpha}
\def \b {\beta}
\def \g {\gamma}
\def \d {\delta}

\def \e {\varepsilon}

\def \l {\lambda}
\def \n {\nabla}

\def \L {\Lambda}

\def \O {\Omega}

\def \cC {\mathcal{C}}
\def \cD {\mathcal{D}}

\def \cG {\mathcal{G}}

\def \cI {\mathcal{I}}

\def \cL {\mathcal{L}}

\def \cR {\mathcal{R}}

\def \cV {\mathcal{V}}

\newcommand{\R}{\ensuremath{\mathbb{R}}}   
\newcommand{\T}{\ensuremath{\mathbb{T}}}   

\def \sign {\mathrm{sgn}}
\def \one {{\mathds{1}}}

\def \p {\partial}

\def \ss {\subset}

\def \dx  {\, \mathrm{d}x}
\def \dt  {\, \mathrm{d}t}

\def \dm  {\, \mathrm{d}m}
\def \dy  {\, \mathrm{d}y}

\def \dr  {\, \mathrm{d}r}
\def \ds  {\, \mathrm{d}s}

\def \dv  {\, \mathrm{d}v}

\def \ddt  {\frac{\mathrm{d\,\,}}{\mathrm{d}t}}

\def \dd  {\mathrm{d}}

\begin{document}

\title{On Cucker-Smale dynamical systems with degenerate communication}

\author{Helge Dietert}

\address{CNRS, Sorbonne Universit\'e, Universit\'e Paris Diderot,
	Institut de Math\'ematiques de Jussieu-Paris Rive Gauche, IMJ-PRG,
	F-75013, Paris, France.}
\email{helge.dietert@imj-prg.fr}

\author{Roman Shvydkoy}
\address{Department of Mathematics, Statistics and Computer Science, University of Illinois at Chicago, 60607}
\email{shvydkoy@uic.edu}

\date{\today}



\thanks{\textbf{Acknowledgment.} Research of RS is supported in part by NSF
	grants DMS 1515705, DMS-1813351 and Simons Foundation. He thanks Courant Institute of Mathematical Sciences for hospitality during the preparation of the paper.}

\begin{abstract}
  This note introduces a new
  method for establishing alignment in systems of collective behavior
  with degenerate communication protocol. The
    communication protocol consists of a kernel defining
    interaction between pairs of agents. Degeneracy presumes that the
    kernel vanishes in a region, which creates a zone of
    indifference. A motivating example is the case of a local
    kernel. Lapses in communication create a lack of coercivity in
    the energy estimates. Our approach is the construction of a
    corrector functional that compensates for this lack of
    coercivity. We obtain a series of new results:
  unconditional alignment for systems on $\R^d$ with degeneracy at
  close range and fat tail in the long range, and for systems on the
  circle with purely local kernels.  The results are proved in the
  context of both the agent based model and its hydrodynamic
  counterpart (Euler alignment model). The method covers bounded and
  singular communication kernels.
\end{abstract}

\maketitle

\tableofcontents
\section{Introduction}
We study long time behavior of Cucker-Smale type systems that describe evolution of a flock driven by laws of self-organization:
\begin{equation}\label{e:CS}
\left\{
\begin{aligned}
\dot{x}_i&=v_i,\\
\dot{v}_i& =\frac1N \sum_{j=1}^N \phi(|x_i - x_j|)(v_j-v_i),
\end{aligned}\right.
\qquad (x_i,v_i)\in \O^d \times \R^d,
\end{equation}
where  $\O^d$  is an environment domain of the flock. The system was  introduced by Cucker and Smale in \cite{CS2007a} to demonstrate a basic mechanism of \emph{alignment} in Galilean invariant settings, i.e. all $v_i \to \bar{v}$, where the limiting velocity is the total momentum  $\bar{v} = \frac1N \sum_{i=1}^{N} v_i$. The original result of Cucker and Smale \cite{CS2007a,CS2007b} (for more refined approach see Ha, Liu \cite{HL2009}, and Tadmor, Tan \cite{TT2014}) claims exponential convergence on $\O^d = \R^d$ for communication kernels satisfying three assumptions:
\begin{itemize}
	\item Monotonicity:  $\phi'(r) \leq 0$;
	\item Non-degeneracy: $\phi(r) >0$;
	\item Fat tail condition: $\int_{1}^\infty \phi(r) \dr = \infty$.
\end{itemize}
It is easy to see that the fat tail is generally a sharp
condition: consider two agents (\(N=2\)) with kernel $\phi(r) = \frac{1}{r^\b}$ for
$r>r_0$, and let $x = x_1 = - x_2>r_0$ and $v = v_1 = -v_2 >0$. This
symmetry is preserved in time. Then the system \eqref{e:CS} becomes
 \[
 \frac{\dx}{2^{\b} x^\b} + \frac{\dv }{1} = 0.
\]
 It admits the conservation law $K = v + \frac{1}{2^{\b}(1-\b) x^{\b-1}}$. If $\b>1$, then $K >0$ and $v(t) \geq c_0$ remains at all times if initially so. This creates permanent misalignment. Monotonicity, on the other hand, can be replaced by the lower bound $\phi \geq \Phi$, where $\Phi$ is decreasing and has fat tail. However, the non-degeneracy is not a trivial assumption to remove and it has been the subject of much of recent research on Cucker-Smale systems. Yet, degenerate kernels appear naturally in many situations. For example, in the swarming dynamics of animals or cell migration, only local communication is possible, $\phi \sim \l \one_{r<r_0}$. In the context of birds or fish motion, a zone of immediate proximity to the object is usually dominated by repulsion, and not alignment, mechanisms. Degenerate kernels appear in the context of a general class of systems with heterophilious communication studied extensively in \cite{MT2014}. In most cases,  emergence of organized behavior is still observed despite lack of global interactions. We refer to these surveys for detailed discussion  \cite{VZ2012,MT2014,Wa2018}.

 Mathematically, obtaining alignment in the case of $\O^d = \R^d$ requires additional assumptions, such as graph-connectivity of the flock \cite{JLM2003,MT2014} or strong local communication \cite{MPT}. Without such special assumptions it is easy to arrange for two agents at a distance $>r_0$ launched in the opposite directions, a configuration that results in permanent misalignment. Under more restrictive environmental settings, such as torus $\O^d = \T^d$, the dynamics is recursive and so, heuristically, one would expect to obtain more generic alignment outcomes. This was achieved for a large class of so-called topological kernels, see \cite{STtopo}, in the context of a macroscopic counterpart of \eqref{e:CS}, known as the  Euler-alignment system,
 \begin{equation}\label{e:CShydro}
 \left\{
 \begin{aligned}
 \p_t \rho +  \n \cdot (v \rho) & = 0\\
 \p_t {v} + v \cdot \n v & = \int_{\O^d} \phi(x-y)\,(v(y) - v(x)) \rho(y) \dy.
 \end{aligned}\right.
\end{equation}
Generally, however, no result for degenerate kernels without special preconditions is known.

In this paper we introduce a new method of proving alignment for degenerate kernels, and use it to obtain a series of new results.  Traditionally, one considers the functionals
\[
\cV_2(t) = \frac{1}{N^2} \sum_{i,j=1}^N |v_{i}(t) - v_j(t)|^2, \quad \cI_2(t) = \frac{2}{N^2} \sum_{i,j=1}^N |v_{i}(t) - v_j(t)|^2 \phi(|x_i - x_j|),
\]
which satisfy the energy equation $\ddt \cV_2 = - \cI_2$.  Clearly, lapses in communication void the usual coercivity estimate $\ddt \cV_2 \leq - c \cV_2$ even if the flock is known to be bounded. The new  idea is to restore communication in the missing range via construction of a corrector. With the use of the corrector we build a family of new Lyapunov functionals for the system -- interesting on its own right -- from which alignment is deduced by means of decay of, possibly, a higher order variation:
\[
\cV_p = \frac{1}{N^2} \sum_{i,j=1}^{N} |v_i - v_j|^p.
\]
In many cases the method gives rates independent of  the number of agents $N$, in which case it carries over to the hydrodynamic system \eqref{e:CShydro}.  Below we summarize the results proved in the paper.

First, we consider two general classes of kernels:  smooth $\phi \in C^1([0,\infty))$, or $\phi \in C^1((0,\infty))$ with non-integrable singularity at the origin:
\begin{equation}\label{e:fh}
\int_0^1 \phi(r) \dr = \infty.
\end{equation}
The significance of condition \eqref{e:fh} lies in the fact that such kernels prevent collisions between agents and consequently, the discrete system \eqref{e:CS} is well-posed even though the right hand side is not Lipschitz. This issue has received extensive treatment in works of Peszek et al  \cite{CCMP2017,MP2018, Pe2014,Pe2015}, and \cite{ST1,ST2,ST3,DKRT2018} for the Euler-alignment system. A quantitative expression of non-collision is an integral part of our approach, so we will revisit the question in Section~\ref{s:coll} below.  A special class of singular kernels satisfying \eqref{e:fh} is described by a power law communication satisfying
\begin{equation}\label{e:power}
\one_{r<r_0} 	\frac{\l}{r^{\b}} \leq \phi(r) \leq \frac{\L}{r^{\b}} , \quad \b\geq 1,\ r_0,\l,\L>0.
\end{equation}
The following is a direct generalization of the classical Cucker-Smale result on $\R^d$.
\begin{theorem}[Alignment on $\R^d$]\label{t:multiD} Suppose the kernel  $\phi \geq 0$ is either smooth or satisfying \eqref{e:fh}. Suppose also that it dominates a monotone fat tail: there exists a non-increasing $\Phi(r)$, such that for some $r_0>0$
	\begin{equation}\label{e:fatquant}
	\phi(r) \geq \Phi(r), \, \forall r>r_0, \quad \text{and} \quad	\int_{r_0}^\infty \Phi(r) \dr = \infty.
	\end{equation}
Then
\begin{itemize}
	\item[(i)] Any solution to the discrete system \eqref{e:CS} aligns: $\cV_2(t) \to 0$, with a rate possibly depending on $N$.
		\item[(ii)] If $\phi$ is smooth, then any solution to the discrete system \eqref{e:CS} aligns: $\cV_4 \to 0$, with a rate independent of $N$.
			\item[(iii)]  If $\phi$ is smooth, then any strong solution to the Euler alignment system  \eqref{e:CShydro} aligns:  $\cV_4 \to 0$.
\end{itemize}
\end{theorem}
In the hydrodynamic case (iii) we use the continuous version of $\cV_p$:
\begin{equation}\label{e:Vkhydro}
\cV_p(t) = \int_{\O^d \times \O^d} |v(x,t) - v(y,t)|^p \rho(x,t) \rho(y,t) \dx \dy.
\end{equation}
The advantage of (i) over (ii) is that it measures alignment in the weaker metric of $\cV_2$, which gives a better (though $N$-dependent) rate over $\cV_4$ in the cases when such rate can be quantified, see \rem{r:ratesopen}.

The next installment concerns local kernels in periodic settings. In multi-D solutions do not always align for degenerate kernels. Indeed, two agents that move parallel to each other  with different velocities, or two agents looping around the torus on perpendicular orbits with relatively rational velocities will not come closer to each other than a distance $r_0$. If that is the radius of communication, then the agents will not align.  All such examples represent measure-zero likelihood, and might be ruled out in various probabilistic settings. We plan to address this problem in another work. However, in 1D, misaligned agents inevitably come close to communicate, and hence the alignment is likely. This is the content of the next result.

\begin{theorem}[Alignment on $\T$]\label{t:T}
	Consider the system \eqref{e:CS} on $\T$ with a non-trivial non-negative kernel $\phi$ which is either smooth or satisfying \eqref{e:fh}. Then the following alignment  holds for any solution:
	\begin{itemize}
		\item[(i)]	$\cV_2(t)  \leq \frac{C N }{t}$, as  $t \to \infty$, where $C$ depends on the initial condition only.
		\item[(ii)] For sub-quadratic singularities
		\begin{equation}\label{e:sub2}
		 \sup_{r<r_0} r^2 \phi(r) <\infty
               \end{equation}
		one has
		\begin{equation}\label{e:V2decay}
		\cV_2(t) \leq C \frac{\ln t}{t}, \quad t \to \infty,
		\end{equation}
		where $C$ depends only on the initial condition.
		\item[(iii)] If the kernel satisfies the more singular assumption \eqref{e:power} for $\b>2$, then
		\begin{equation*}
		\cV_2(t) \to 0, \quad t \to \infty,
		\end{equation*}
		with a rate independent of $N$.
	\end{itemize}
In the cases $(ii)$ and $(iii)$ in the range $2<\b<3$, the result holds for any strong solution to the Euler-alignment system \eqref{e:CShydro} as well.
\end{theorem}
In the hydrodynamic settings the alignment is measured in $L^2$-metric, and as such does not imply absolute alignment in $L^\infty$ norm. The letter, however, is natural due to the maximum principle which at least shows that  the amplitude $(\max v - \min v)$ is non-increasing.  We provide a separate argument to address alignment in $L^\infty$ settings, based on dynamical tracking of particle trajectories.
\begin{theorem}\label{t:1DTcont}
	Consider the system \eqref{e:CShydro} on $\T$ with a smooth non-trivial non-negative kernel. Then any global classical solution aligns:
	\begin{equation}
	\sup_{x}|v(x,t) - \bar{v}| \leq  C \left(  \frac{\ln t}{t}\right)^{\frac15},
	\end{equation}
	where
	\[
	\bar{v} =\frac{1}{M} \int_\T v \rho \dx, \quad    M = \int_\T \rho(x,t) \dx.
	\]

\end{theorem}
Note that no connectivity of the flock is assumed in either of the theorems, i.e. the solution may have vacuum initially. This improves in 1D case the result obtained for topological kernels \cite{STtopo} where non-vacuous datum is essential.

In this paper we completely omit discussion of regularity of the hydrodynamic system \eqref{e:CShydro}. All our results pertaining to \eqref{e:CShydro} hold \emph{for a given strong solution}. We note, however, that in a variety of situations, both for smooth and singular kernels, such solutions have been constructed, see \cite{CCTT2016,DKRT2018,ST1,ST2,ST3,TT2014,Tan2017}. In particular, the 1D case is completely understood,  and small initial data results are known in multi-D case, \cite{DMPW2018,HKK2014,HeT2017,TT2014,S-sid}.

\section{Preliminaries}

In what follows we use the following  abbreviations:
\[
v_{ij} = v_i - v_j,\, x_{ij} = x_i - x_j,\, \phi_{ij} = \phi(x_{ij}),\quad  \text{etc.}
\]
Let us define the following variation functionals
\[
\cV_p = \frac{1}{N^2} \sum_{i,j=1}^{N} |v_{ij}|^p, \quad \cI_p = \frac{p}{N^2} \sum_{i,j=1}^{N} |v_{ij}|^p \phi_{ij}, \quad  p\geq 1\added{.}
\]
We observe that $\cV_p$'s are  non-increasing. Indeed,
\[
\begin{split}
\ddt \cV_p &=  \frac{p}{N^3} \sum_{i,j,k}  |v_{ij}|^{p-2} v_{ij} \cdot (v_{ki}  \phi_{ki}  - v_{kj}  \phi_{kj}) =  \frac{2p}{N^3} \sum_{i,j,k}  |v_{ij}|^{p-2} v_{ij} \cdot v_{ki}  \phi_{ki}  \\
& =  \frac{p}{N^3} \sum_{i,j,k} ( |v_{ij}|^{p-2} v_{ij} -  |v_{kj}|^{p-2} v_{kj}) \cdot v_{ki}  \phi_{ki} =  \frac{p}{N^3} \sum_{i,j,k} ( |v_{ij}|^{p-2} v_{ij} -  |v_{kj}|^{p-2} v_{kj}) \cdot (v_{kj} - v_{ij})  \phi_{ki},
\end{split}
\]
with the convention that $ |v_{ij}|^{p-2} v_{ij}  = 0$ if $i=j$. The right hand side is non-positive due to the elementary inequality
\[
\left(|a|^{p-2} a -  |b|^{p-2} b \right)  \cdot (a -b)  \geq 0.
\]
The two special cases, $i=j$ and $k=j$, produce the term $- |v_{ik}|^p -|v_{ij}|^p$. So, in general we have an $N$-dependent inequality
\begin{equation}\label{e:Vp}
\ddt \cV_p \leq  - \frac{1}{N} \cI_p, \quad p\geq 1.
\end{equation}
For the case $p=2$ we have the $N$-independent identity (energy law):
\begin{equation}\label{e:V2}
\ddt \cV_2 = -  \cI_2.
\end{equation}

\subsection{Collisions under singular communication} \label{s:coll}The calculations above are valid for smooth kernels or when no collisions are present. In this section we will collect all necessary facts about the collision issue.  We start with two short examples to illuminate the basic mechanism (see also \cite{Pe2014}).

\begin{example}
	First, in the smooth case let us assume $\phi = 1$ in a neighborhood of $0$ for simplicity. Let us arrange two agents $x= x_1 = -x_2$ with $0< x(0) = \e \ll1$. And let $v_1 = -v_2 < 0$ be very large. Clearly $x(t)$ will remain in the same neighborhood of $0$ as where it has started, and so the system reads
	\[
	\ddt x = v, \quad \ddt v = -2v.
	\]
	Solving it explicitly we can see that the two agents will collide at the origin. In the smooth kernel case such hard interactions are common but present no difficulty for well-posedness simply because the right hand side of \eqref{e:CS} remains smooth.
\end{example}
\begin{example}
	 Singularity, obviously, presents a problem if agents collide. However if it is sufficiently strong the agents align faster than a collision has a chance to happen.  Just how much singularity is necessary  can be seen from an example similar to the one given in the Introduction.  Let the kernel be given by $\phi(r) = \frac{1}{r^\b}$, and let us consider the same setup as previously. Then we obtain the system
	\[
	\ddt x = v, \quad \ddt v = -\frac{v}{2^\b x^{\b}}.
      \]
	This system has a conservation law:  $v + \frac{x^{1-\b}}{(1-\b) 2^\b} = K$. So, provided $\b<1$, and initially $K \ll 0$, then $v< K \ll 0$ as well. This means that $x$ will reach the origin in finite time.
\end{example}

\begin{theorem}
	Under the strong singularity condition \eqref{e:fh} the flock experiences no collisions between agents for any non-collisional initial datum on $\R^d$ or $\T^d$. Consequently, any non-collisional initial datum gives rise to a unique global solution.
\end{theorem}
\begin{proof} We start as in \cite{CCMP2017}. 	Let us assume that for a given non-collisional initial condition $(x_{i}, v_{ i} )_{i}$ a collision occurs at time $T^*$ for the first time. Let $I^* \ss I= \{1,...,N\}$ be one set of indexes of the  agents that collided at one point (note that other groups of agents may collide as well at other points of space). Hence, there exists a $\d>0$ such that $|x_{ik}(t)| \geq \d$ for all $i \in I^*$ and $k \in \O \backslash   I^*$. Denote
	\[
	\cD^* (t) =  \max_{i,j \in I^*} |x_{ij}(t)|, \quad \cV^*(t) = \sum_{i,j \in I^*} |v_{ij}(t)|^2 .
	\]
	Directly from the characteristic equation we obtain $| \dot{\cD}^* | \leq \sqrt{\cV^*}$, and hence
	\begin{equation}
	- \dot{\cD}^* \leq \sqrt{ \cV^*}.
	\end{equation}
	From the momentum equation we obtain
\[
	\ddt \cV^* = \frac{2}{N} \sum_{k \in I, i,j\in I^*} \phi_{ik} v_{ k i} v_{ ij}  - \phi_{kj}v_{ kj} v_{ ij} .
\]
	Switching $i,j$ in the second sum results in the same as first sum. So, we obtain
	\[
	\ddt  \cV^*  \leq \frac{4}{N} \sum_{k \in I, i,j\in I^*} \phi_{ki} v_{ ki} v_{ij} =  \frac{4}{N} \sum_{k \in I \backslash  I^*, i,j\in I^*} \ldots +  \frac{4}{N}\sum_{k, i,j\in I^*} \ldots
	\]
	Note that in the first sum the agents are separated, and all velocities are bounded by the maximum principle. So, we can estimate it by $C_1 \sqrt{\cV^*}$. Symmetrizing the second sum on $k,i$ we obtain
	\[
	\begin{split}
	\ddt  \cV^*& \leq C_1 \sqrt{\cV^*}+ \frac2N \sum_{k, i,j\in I^*} \phi_{ki}v_{ ki} v_{ ik}  \leq   C_1 \sqrt{\cV^*}- C_2 \phi(\cD^*)  \cV^*.
	\end{split}
	\]
        We thus obtain a system
	\begin{equation}
	\left\{
	\begin{aligned}
	\ddt \sqrt{\cV^*}& \leq  - C_2 \phi(\cD^*) \sqrt{\cV^*}+ C_1\\
	- \ddt{\cD}^* &\leq  \sqrt{\cV^*} .
	\end{aligned}\right.
	\end{equation}
	Let us denote the energy functional
	\[
	E(t) = \sqrt{\cV^*(t)} + C_2 \int_{\cD^*(t)}^1 \phi(r) \dr.
	\]
	We readily find that $\ddt E \leq C_1$, hence $E$ remains bounded up to the critical time. This means that $\cD^*(t)$ cannot approach zero value.

	The global existence part is now a routine application of the Picard iteration and the standard continuation argument.
\end{proof}

A quantitative version on the minimal distance between agents was obtained in \cite{CCMP2017} in the case of power kernels \eqref{e:power} with $\b \geq 2$, namely $|x_{ij}(t)| \geq c e^{-Ct}$. We can considerably improve this bound in the course of the computation below that will be needed in the proof of \thm{t:T}.

So, let us consider  the local version of the collision functional introduced in \cite{CCMP2017} for $\b\geq 2$:
\begin{equation}\label{e:C}
\cC =
\begin{dcases}
\frac{1}{N^2} \sum_{i,j=1}^{N} \frac{1}{(|x_{ij}|\wedge r_0)^{\b-2}}, & \quad \b >2 \\
\frac{1}{N^2} \sum_{i,j=1}^{N} \ln(|x_{ij}|\wedge r_0), & \quad \b =2.
\end{dcases}
\end{equation}
Note that as long as \(b<3\) this will stay finite in the limit \(N\to\infty\) for a smooth
distribution. Let us now look at the time-derivative of $\cC$ for $\b > 2$:
\begin{equation*}
\begin{aligned}
\frac{\dd \cC}{\dd t}
&= \frac{(2-\b)}{N^2} \sum_{i,j=1}^{N}
\frac{\ddt
	\Big(|x_{ij}|\wedge r_0\Big) }{(|x_{ij}|\wedge r_0)^{\b-1}} \le
\frac{|\b-2|}{N^2} \sum_{i,j=1}^{N}
\,\frac{1}{(|x_{ij}|\wedge r_0)^{\b-1}}
\,|v_{ij}|\, \one_{|x_{ij}|<r_0}\\
&\le |\b-2|\,
\left(
\frac{1}{N^2}
\sum_{i,j=1}^{N}
v_{ij}^2 \frac{1}{|x_{ij}|^\b}\, \one_{|x_{ij}|<r_0}
\right)^{1/2}
\left(
\frac{1}{N^2}
\sum_{i,j=1}^{N}
|x_{ij}|^{2-\b}\, \one_{|x_{ij}|<r_0}
\right)^{1/2}\\
&\le C
\sqrt{\cI_2}\, \sqrt{\cC}.
\end{aligned}
\end{equation*}
This implies
\begin{equation}\label{e:Cbound}
\sqrt{\cC(t)}
\le \sqrt{\cC(0)} + C \int_0^t \sqrt{\cI_2(s)} \ds,
\end{equation}
and recalling  that $\cI_2$ is integrable on $\R^+$ we conclude
\begin{equation}\label{e:linearC}
\cC(t) \lesssim t.
\end{equation}
For $\b=2$, exact same computation gives $\ddt \cC \leq C \sqrt{\cI_2}$, hence $\cC(t) \lesssim \sqrt{t}$. We thus arrive at the following bounds
\begin{equation}
	|x_{ij}(t)| \geq
	\begin{dcases}
	\frac{c}{t^{\frac{1}{\b-2}}}, & \quad \b>2, \\
	c e^{-C \sqrt{t}}, & \quad \b=2.
	\end{dcases}
\end{equation}

\subsection{Lagrangian coordinates in hydrodynamic context}

In the hydrodynamic context of  \eqref{e:CShydro} we express the variation functionals in terms of the density measure
\[
\dm_t = \rho(x,t) \dx.
\]
In view of the continuity equation, this measure is transported along the flow of $v$. Namely, letting
\[
\begin{split}
\ddt x(\a,t) & = v(x(\a,t),t), \quad t>0 \\
x(\a,0 ) & = \a,
\end{split}
\]
the measure $\dm_t$ becomes a push-forward  of $\dm_0$ under the flow map $x(\cdot,t)$:
$\dm_t = x(\cdot,t) \# \dm_0$. In other words, for any $f$,
\begin{equation}
\int_{\O^d}f( x(\a,t)) \dm_0(\a) = \int_{\O^d}f( x) \dm_t(x) .
\end{equation}
We now define the hydrodynamic variations
\begin{equation}\label{e:VpIp}
\begin{split}
\cV_p(t) & = \int_{\O^d \times \O^d} |v(x,t) - v(y,t)|^p\rho(x,t) \rho(y,t) \dx \dy \\
\cI_p(t) & = p\int_{\O^d \times \O^d} |v(x,t) - v(y,t)|^p \rho(x,t) \rho(y,t) \phi(x-y) \dx \dy.
\end{split}
\end{equation}
As before, we have the energy law
\[
\ddt \cV_2  = - \cI_2,
\]
and monotonicity of all other functionals $\ddt \cV_p \leq 0$.

We will work in Lagrangian coordinates in what follows and use the short notation
\[
u(\a,t) = v(x(\a,t),t), \quad u(\a,t) - u(\b,t) = u_{\a\b}, \quad x(\a,t) - x(\b,t) = x_{\a\b}, \quad \phi(x(\a,t) - x(\b,t)) = \phi_{\a\b},
\]
and $\dm_0(\a,\b) = \dm_0(\a) \dm_0(\b)$ to denote  the product measure. Thus,
\[
\cV_p(t) = \int_{\O^d \times \O^d} |u_{\a\b}|^p  \dm_0(\a,\b).
\]

\section{Alignment  on $\R^d$}\label{s:multiD}

We first demonstrate our method in the context of the multidimensional open environment $\O^d = \R^d$.  Let us note that the alignment process may in fact be very slow. If all initial differences $x_{ij}$ fall in the indifference zone $r<r_0$, the dispersion itself undergoes relaxation on a linear time scale.  So, we do not expect a fast rate in general.

\begin{proof}[Proof of \thm{t:multiD} (i)]
	Let us consider the following  ``distance" -- the projection of the displacement $x_{ij}$ onto the relative velocity direction taken with opposite sign:
	\begin{equation}\label{e:dmultiD}
	d_{ij} =  - x_{ij} \cdot \frac{v_{ij}}{|v_{ij}|}.
	\end{equation}
Next, define two auxiliary functions \(\chi : \R^+ \mapsto \R^+\) and
	\(\psi : \R \mapsto \R^+\) by
	\begin{equation*}
	\chi(r) =
	\begin{dcases}
	1,& r < r_0, \\
	2-\frac{r}{r_0},& r_0 \le r \le 2r_0, \\
	0,& r > 2r_0
	\end{dcases}
	\quad \text{and} \quad
	\psi(x) =
	\begin{dcases}
	0, & x<-r_0,\\
	x+r_0, & |x|\le r_0, \\
	2r_0, & x>r_0.
	\end{dcases}
	\end{equation*}
	We consider the following corrector function with spacial
        truncation
	\[
	\cG = \frac{1}{N^2} \sum_{i,j} |v_{ij}| \psi(d_{ij}) \chi(|x_{ij}|),
	\]
	and compute its derivative:
	\begin{equation}\label{e:Gder}
	\ddt \cG = - \frac{1}{N^2} \sum_{i,j} |v_{ij}|^2 \one_{|d_{ij}|<r_0}  \chi(|x_{ij}|) + \cR_1 + \cR_2 + \cR_3,
	\end{equation}
	where
	\[
	\begin{split}
	\cR_1 & =  \frac{2}{N^3} \sum_{i,j,k} \frac{v_{ij}}{|v_{ij}|} \cdot v_{ki} \phi(x_{ik})   \psi(d_{ij}) \chi(|x_{ij}|), \\
	\cR_2 & = -\frac{2}{N^3} \sum_{i,j,k} x_{ij} \cdot \left( \Id - \frac{v_{ij} \otimes v_{ij}}{|v_{ij}|^2} \right) v_{ki} \phi(x_{ik}) \one_{|d_{ij}|<r_0}   \chi(|x_{ij}|) \\
	\cR_3 & =  \frac{1}{N^2} \sum_{i,j} |v_{ij}| \psi(d_{ij}) \chi'(|x_{ij}|) \frac{x_{ij}}{|x_{ij}|} \cdot v_{ij}.
	\end{split}
	\]
	As seen from the first gain term in \eqref{e:Gder}, the role of $\cG$ is compensate for missing communication in the close range. Let us address this term first. Without loss of generality we can assume that $\Phi$ is a bounded function on $\R^+$, and hence  we have
	\[
	\one_{r<r_0}  + \phi(r) \geq c \Phi(r),
	\]
	for all $r >0$ and some $c>0$.  Using that $\one_{|d_{ij}|\geq r_0} \leq \one_{|x_{ij}|\geq r_0}$ we obtain
	\[
	\begin{split}
	&- \frac{1}{N^2} \sum_{i,j} |v_{ij}|^2 \one_{|d_{ij}|<r_0}  \chi(x_{ij})  \leq   - \frac{1}{N^2} \sum_{i,j} |v_{ij}|^2  \one_{|x_{ij}|<r_0}  \\
	&= - \frac{1}{N^2} \sum_{i,j} |v_{ij}|^2 ( \one_{|x_{ij}|<r_0}  + \phi(x_{ij}) )+  \frac{1}{N^2} \sum_{i,j} |v_{ij}|^2  \phi(x_{ij}) \\
	& \leq - c \Phi(\cD) \cV_2 + \cI_2,
	\end{split}
	\]
	where $\cD = \cD(t)$ is the diameter of the flock.

	Let us proceed to the error terms.  By construction of the functions,
	\[
	|\cR_1|, |\cR_2| \lesssim \cI_1,
	\]
	and
	\[
	\cR_3 \leq  \frac{1}{N^2} \sum_{i,j} |v_{ij}|^2 | \chi'(x_{ij})|  \leq \frac{1}{N^2} \sum_{i,j} |v_{ij}|^2 \one_{r_0 < |x_{ij}| < 2r_0} \lesssim  \frac{1}{N^2} \sum_{i,j} |v_{ij}|^2  \phi(x_{ij}) = \cI_2.
	\]
	Hence, we obtain, for constants \(a,b,c\) only depending on
        \(\phi\), that
	\[
	\ddt \cG \leq - c \Phi(\cD) \cV_2 + a \cI_2 + b \cI_1.
	\]
	Let us form the following functional
	\[
	\cL = \cG + a \cV_2+ b N \cV_1.
	\]
	Then $\ddt \cL \leq - c \Phi(\cD) \cV_2$, and hence
	\[
	\int_0^\infty \Phi(\cD(t)) \cV_2(t) \dt < \infty.
	\]
	Next, let us note that due to uniform bound on velocity, $ \cD(t) \leq c t + \cD_0$, so
	\begin{equation}\label{e:FV2}
	\int_0^\infty \Phi(c t + \cD_0)  \cV_2(t) \dt< \infty.
	\end{equation}
	Due to non-integrability of $\Phi$, $\cV_2$ cannot stay bounded away from the origin. Hence, due to its monotonicity, $\cV_2 \to 0$.

\end{proof}

\begin{proof}[Proof of \thm{t:multiD} (ii)]
	Here we consider the third order corrector function
	\begin{equation*}
	\cG_3 = \frac{1}{N^2} \sum_{i,j} |v_{ij}|^3  \psi(d_{ij})\chi(|x_{ij}|),
	\end{equation*}
	with $\psi$ and $\chi$ defined in the previous proof.
	We  find
	\begin{equation*}
	\ddt \cG_3 = - \frac{1}{N^2} \sum_{i,j=1}^{N} |v_{ij}|^4 \one_{|d_{ij}|<r_0}  \chi_{ij} + \cR_1 + \cR_2 + \cR_3,
	\end{equation*}
	with
	\begin{align*}
	\cR_1 & =  \frac{6}{N^3} \sum_{i,j,k=1}^{N} |v_{ij}| v_{ij} \cdot
	v_{ki}\, \phi_{ik}\, \psi_{ij} \chi_{ij}, \\
	\cR_2 & = \frac{-2}{N^3} \sum_{i,j,k=1}^{N} |v_{ij}|^2\, x_{ij} \cdot
	\left( \Id - \frac{v_{ij} \otimes v_{ij}}{|v_{ij}|^2} \right)
	v_{ki} \phi_{ki}\,
	\one_{|d_{ij}|<r_0}   \chi_{ij} \\
	\cR_3 & =  \frac{1}{N^2} \sum_{i,j=1}^{N} |v_{ij}|^3 \psi_{ij} \chi'(x_{ij}) \frac{x_{ij}}{|x_{ij}|} \cdot v_{ij}.
	\end{align*}
	The gain term is estimated as before using that we have an a priori uniform bound on all velocities $|v_i(t)| \leq |v(0)|_\infty$:
	\begin{equation}\label{e:gain2}
	- \frac{1}{N^2} \sum_{i,j=1}^{N} |v_{ij}|^4 \one_{|d_{ij}|<r_0}  \chi_{ij} \leq  - c \Phi(\cD) \cV_4 + \cI_2.
	\end{equation}
	Next, again as before,
	\begin{equation*}
	\cR_3 \lesssim  \frac{1}{N^2} \sum_{i,j} |v_{ij}|^3  \phi(x_{ij})  \lesssim \cI_2.
	\end{equation*}
	By Young's inequality we find for \(\epsilon > 0\)
	\begin{equation*}
	\cR_2 \lesssim  \frac{\e}{N^2} \sum_{i,j=1}^{N}
	|v_{ij}|^4 |x_{ij}|^2 \chi_{ij}^2
	+ \frac{|\phi|_\infty}{\e}  \frac{1}{N^2}
	\sum_{i,k=1}^{N} |v_{ki}|^2 \phi_{ki}.
	\end{equation*}
	However, \(\chi_{ij}\) is vanishing if \(|x_{ij}| > 2r_0\) so that we
	find
	\begin{equation*}
	\begin{split}
	\cR_2&  \lesssim \frac{\e}{N^2} \sum_{i,j=1}^{N}
	|v_{ij}|^4 \one_{|x_{ij}| \le 2r_0}
	+  \frac{1}{\e N^2}
	\sum_{i,k=1}^{N} |v_{ki}|^2 \phi_{ki}\\
	& \leq \frac{\e}{N^2} \sum_{i,j=1}^{N}
	|v_{ij}|^4 \one_{ |x_{ij}| \leq r_0}+ \frac{\e}{N^2} \sum_{i,j=1}^{N}
	|v_{ij}|^4 \one_{r_0< |x_{ij}| < 2r_0}
	+  \frac{1}{\e N^2}
	\sum_{i,k=1}^{N} |v_{ki}|^2 \phi_{ki}
	\end{split}
	\end{equation*}
	By choosing $\e$ small enough the first sum is absorbed into the gain term \eqref{e:gain2}. The second sum is dominated by $\cI_2$, and so is the last third sum.

	The term \(\cR_1\) can be estimated in exact same manner. Thus,
	\begin{equation*}
	\ddt \cG_3
	\le - c \Phi(\cD) \cV_4 + a \cI_2.
	\end{equation*}
	With the help of the functional $\cL = \cG_3 + a \cV_2$ we conclude that
	\begin{equation}\label{e:FV4}
	\int_0^\infty \Phi(\cD(t)) \, \cV_4(t) \dt < \infty,
	\end{equation}
	with the integral bound being independent of $N$.
	The rest of the argument proceeds as before.
\end{proof}

The proof of  \thm{t:multiD} (iii) repeats that of part (ii) line by line with the use of continuous quantities throughout:
\[
d_{\a\b} =  - x_{\a\b} \cdot \frac{u_{\a\b}}{|u_{\a\b}|}, \quad 	\cG_3 = \int_{\R^{2d}} |u_{\a\b}|^3  \psi(d_{\a\b})\chi(|x_{\a\b}|) \dm_0(\a,\b), \quad \text{etc}.
\]

\begin{remark}\label{r:ratesopen}
	In either of the results above we can actually extract a specific rate of alignment if we make an explicit assumption about the tail of the kernel. Thus, if
	\[
	\phi(r) \sim \frac{1}{r^\b}, \quad \forall r>r_0,
	\]
	and $\b\leq 1$, then from either \eqref{e:FV2} or \eqref{e:FV4},
	\[
	\int_1^\infty \frac{1}{t^{\b}} \, \cV(t) \dt < \infty,
	\]
	where $\cV$ is the corresponding functional. So, if $\b=1$, then
	there exists an $A>0$ such that for any $T>1$ there exists a $t\in [T,T^A]$ such that
	\begin{equation}\label{e:Vlog}
	\cV(t) < \frac{1}{\ln t}.
	\end{equation}
	Since $\ln t$ is proportional to $\ln T$ for all $t\in [T,T^A]$ this proves  \eqref{e:Vlog} for all large times. If $\b<1$, then we argue that for some large $A>0$ and all $T>0$ we find $t\in [T,AT]$ such that $\cV(t) \leq \frac{1}{t^{1-\b}}$. But the latter is comparable for all values of $t\in [T,AT]$. Thus we obtain the power rate of
	\begin{equation}
	\cV(t) \leq \frac{1}{t^{1-\b}}.
	\end{equation}
\end{remark}

\section{Alignment on $\T$}

Let us now turn to the case of a 1D torus, another environment where our method can be successfully applied. First, we address the discrete systems.

\subsection{Discrete systems}

\begin{proof}[Proof of \thm{t:T} (i)] Since the kernel is non-trivial, it will satisfy
	\begin{equation}\label{e:lowI}
		 \phi \geq \l \one_I,
	\end{equation}
	for some subinterval $I \ss (-\pi,\pi)$. For the sake of definiteness, we will run the proof for the centrally located interval $I = [-r_0,r_0]$, and will explain how to modify it to suit the general case.

	First, we present a quick dynamical argument that sets the intuition behind the result. By the Galilean invariance we can assume that the momentum vanishes, $\sum_{i=1}^{N} v_i = 0$. Since $\int_0^\infty \cI_1 \ds < \infty$ for a fixed $\e>0$ there is a time $T$ such that
	\[
	\int_{T}^\infty \cI_1(s) \ds < \e.
	\]
This immediately implies that velocity variations from that time on remain small:
\[
| v_i(t+T) - v_i(T)| \leq \e.
\]
Let  $U$ be the maximal velocity at time $T$, and assume it is large: $U>2\e$. Let $v_i(T) = U$ for some $i$. Noting that $v_i(t+T) > U/2$ for all $t>0$ we conclude
\[
x_i(t+T) > x_i(T) + t U/2 = x_i(T) + 4\pi,
\]
for $t=8\pi/U$.  At the same time, since the total momentum is zero, we can find $j$ with $v_j(T)<0$, hence
\[
 x_j(t+T) < x_j(T) + t \e < x_j(T) + \e 8\pi/U < x_j(T)  + \pi,
\]
provided $U> 8\pi \e$. So, unless $U\leq 8\pi \e$, during this time the two agents will collide, which is a contradiction in the singular case. In the smooth case, we argue that the agents will remain $r_0$-close on a time span $t''-t' = r_0/U$. However, in this case
\[
\l \int_{t'}^{t''} |v_i(t) - v_j(t)| \dt \leq \int_{t'}^{t''} \cI_1(t) \dt < \e.
\]
This implies that at some time $t \in [t',t'']$, $|v_i(t) - v_j(t)| \lesssim \e$, which brings us to  $U\lesssim \e$.   In either case we arrive at the same conclusion  $U\lesssim \e$.

This argument does not give a good quantitative estimate on the rate of decay. Instead we go back to the basic energy laws \eqref{e:Vp} - \eqref{e:V2} and construct a corrector functional $\cG$ which serves to compensate for the missing non-local interactions.

First, we define a periodic analogue of the directed distance:
\[
d_{ij}(t) = - x_{ij}\, \sign(v_{ij}) \mod 2\pi\replaced{,}{.}
\]
where $x_i,x_j \in [0,2\pi)$ are viewed on the same coordinate chart. The distance picks up the length of the arch between $x_i$ and $x_j$ which dynamically contracts at time $t$ under the evolution of the two agents.  The distance clearly undergoes jump discontinuities at $x_i = x_j$ and $v_i = v_j$. Otherwise, we have
\begin{equation}\label{e:dd}
\frac{\dd}{\dd t} d_{ij} = - |v_{ij}|.
\end{equation}
Next we define an auxiliary communication kernel $\psi \geq 0$ as follows
\[
\psi(x) =
\begin{dcases}
-x + r_0, & \quad -r_0 \leq x \leq r_0 \\
\frac{r_0}{\pi-r_0} x - \frac{r_0^2}{\pi-r_0}, & \quad r_0< x<2\pi-r_0,
\end{dcases}
\]
extended periodically on the line. With this we define the corrector
\begin{equation*}
\cG(t) = \frac{1}{N^2} \sum_{i,j=1}^{N} |v_{ij}| \, \psi(d_{ij}).
\end{equation*}
Let us observe the formula for the derivative of $\cG$:
\[
\ddt \cG =  -\frac{1}{N^2} \sum_{i,j=1}^{N} |v_{ij}|^2 \psi'(d_{ij})
+ \frac{2}{N^3} \sum_{i,j=1}^{N} \psi(d_{ij}) \sign(v_{ij})
\sum_{k=1}^{N}
v_{ki}  \phi_{ki}.
\]
The formula can be justified classically, at those times when there is not jump, i.e. $x_i \neq x_j$ and $v_i \neq v_j$, due to \eqref{e:dd}. When two agents pass each other $x_i = x_j$ we use periodicity of $\psi$, and when $v_{ij} = 0$, the pre-factor $|v_{ij}|$ vanishes.

From this we continue
	\begin{equation*}
	\frac{\dd}{\dd t} \cG(t)=  \frac{1}{N^2} \sum_{i,j=1}^{N}
	|v_{ij}|^2\, \one_{|x_{ij}| \leq r_0}
	- \frac{r_0}{\pi-r_0} \frac{1}{N^2} \sum_{i,j=1}^{N}
	|v_{ij}|^2\, \one_{|x_{ij}| \geq r_0}+ \cR,
	\end{equation*}
where
	\begin{equation}\label{e:Rest}
	\cR = \frac{2}{N^3} \sum_{i,j,k=1}^{N} \psi(d_{ij}) \sign(v_{ij})	v_{ki}  \phi_{ki} \leq c \cI_1.
	\end{equation}
So, we obtain
\[
	\frac{\dd}{\dd t} \cG(t) \leq  \frac{\pi}{c_0(\pi-r_0)} \cI_2  - \frac{r_0}{\pi-r_0} \cV_2	 + c \cI_1 = a \cI_2 - b \cV_2 + c \cI_1.
\]
With this we can form another Lyapunov functional:
\[
\cL = \cG + \frac{c N}{2} \cV_1 + b t \cV_2 + a  \cV_2,
\]
so that $\ddt \cL \leq 0$.  This implies the conclusion of the theorem immediately.

A modification of the proof under more general lower bound \eqref{e:lowI} is only required in the definition of $\psi$, where we place the negative linear slope exactly over $I$, and the rest reconnects with positive slope.
\end{proof}

\begin{proof}[Proof of \thm{t:T} (ii)] The previous proof is repeated  up to the remainder estimate \eqref{e:Rest} which  we will replace using the collision potential \eqref{e:C} in order to avoid the use of $N$-dependent $\cI_1$. Symmetrizing over $i,k$, we find
	\begin{equation*}
	\cR = \frac{1}{N^3}
 \sum_{i,j,k=1}^{N} ( \psi(d_{ij}) \sign(v_{ij}) - \psi(d_{kj}) \sign(v_{kj}))	v_{ki}  \phi_{ki}.
	\end{equation*}
	If \(v_i \ge v_j \ge v_k\) or \(v_i \le v_j \le v_k\), then the
	summand is negative, and so we can drop it. Hence,
	\begin{equation*}
	\begin{split}
	\cR &\le \frac{1}{N^3}
	\sum_{i,j,k=1}^{N} \Big( \psi(d_{ij}) \sign(v_{ij}) - \psi(d_{kj}) \sign(v_{kj})\Big)	v_{ki}  \phi_{ki}\,
	\one_{v_j > \max(v_i,v_k)}\\
	&\quad+\frac{1}{N^3}
	\sum_{i,j,k=1}^{N} \Big( \psi(d_{ij}) \sign(v_{ij}) - \psi(d_{kj}) \sign(v_{kj})\Big)	v_{ki}  \phi_{ki}\,
	\one_{v_j < \min(v_i,v_k)}\\
	& = \frac{1}{N^3}
	\sum_{i,j,k=1}^{N} \Big( \psi(d_{kj}) - \psi(d_{ij})  \Big)	v_{ki}  \phi_{ki}\,
	\one_{v_j > \max(v_i,v_k)}\\
	&\quad+\frac{1}{N^3}
	\sum_{i,j,k=1}^{N} \Big( \psi(d_{ij}) - \psi(d_{kj}) \Big)	v_{ki}  \phi_{ki}\,
	\one_{v_j < \min(v_i,v_k)}.
	\end{split}
	\end{equation*}
	In the considered cases \(v_j > \max(v_i,v_k)\) and
	\(v_j < \min(v_i,v_k)\), we see that \(v_i-v_j\) and \(v_k-v_j\) have
	the same sign so that \(d_{ij}\) and \(d_{kj}\) are computed in the
	same direction. Hence we find by the Lipschitz continuity of
	\(\psi\) and by the triangle inequality that
	\(|\psi(d_{ij}) - \psi(d_{kj})| \le C |x_i-x_k|\). Therefore,
	\begin{equation*}
	\cR \le \frac{C}{N^3}
	\sum_{i,j,k=1}^{N}	|x_{ik}|v_{ik}\phi_{ki} =  \frac{C}{N^2}
	\sum_{i,k=1}^{N}	|x_{ik}|v_{ik}\phi_{ki}  \leq  \frac1t \frac{C}{b N^2}
	\sum_{i,k=1}^{N}	|x_{ik}|^2 \phi_{ki} + t \frac{b}{N^2}
	\sum_{i,k=1}^{N}	v_{ik}^2\phi_{ki} .
	\end{equation*}
Let us 	proceed now under the assumptions of (ii). Here we obtain
\[
\cR  \leq \frac{c}{t} + b t \cI_2.
\]
Then the corrector equation becomes
\[
\frac{\dd}{\dd t} \cG(t) \leq  a \cI_2 + b( t \cI_2 -  \cV_2 ) + \frac{c}{t}.
\]
With this we can form another functional:
\[
\cL = \cG + b t \cV_2 + a \cV_2,
\]
satisfying $\ddt \cL \leq  \frac{c}{t}$. Hence, $\cL(t) \lesssim \ln t$, and the resulting bound is as claimed.

\end{proof}

\begin{proof}[Proof of \thm{t:T} (iii)] For part  (iii) we use the collision potential:
\[
\cR \leq C \left(\frac{1}{N^2}
\sum_{i,k=1}^{N}	|x_{ik}|^{2-\b} |x_{ik}|^\b \phi_{ki} \right)^{1/2}\sqrt{ \cI_2} \lesssim \sqrt{ \cI_2} \sqrt{\cC}.
\]
Recalling  \eqref{e:Cbound},
\[
\cR \leq c_1\sqrt{ \cI_2(t)} + c_2 \sqrt{ \cI_2(t)} \int_0^t \sqrt{\cI_2(s)} \ds.
\]
We can replace as before
\[
 c_1\sqrt{ \cI_2(t)} \leq \frac{c_3}{t} + b t \cI_2(t),
 \]
 obtaining
\[
\frac{\dd}{\dd t} \cG(t) \leq  a \cI_2 + b( t \cI_2 -  \cV_2 ) + \frac{c_3}{t} + c_2 \sqrt{ \cI_2(t)} \int_0^t \sqrt{\cI_2(s)} \ds.
\]
With $\cL$ defined as before,
\[
\ddt \cL \lesssim  \frac{c_3}{t} + c_2 \sqrt{ \cI_2(t)} \int_0^t \sqrt{\cI_2(s)} \ds.
\]
Integrating,
\[
\cL(T) \lesssim \cL(0) + \ln T + \left( \int_0^T \sqrt{\cI_2(s)} \ds \right)^2.
\]
Hence,
\[
\cV_2(T) \leq \frac1T \cL(T) \lesssim \frac{\ln T}{T} + \frac1T \left( \int_0^T \sqrt{\cI_2(s)} \ds \right)^2.
\]
The right hand side obviously tends to zero which can be seen by splitting the integral into $(0,T')$ and $(T',T)$, where $T'$ is large.
\end{proof}

\subsection{Hydrodynamic systems}
In this section we establish hydrodynamic  analogues of \thm{t:T} part (ii), and part (iii)  with $2<\b<3$,  and prove \thm{t:1DTcont}.

\begin{proof}[Proof of  \thm{t:T} part (ii), and part (iii)]
Proceeding as in the discrete case we define the directed distance
\[
d_{\a\b} (t) = (x(\a,t) - x(\b,t))\, \sign(u(\b,t) - u(\a,t)) \mod 2\pi,
\]
and the corrector with $\psi$ as before:
\[
\cG = \int_{\T^2} |u_{\a\b}| \psi(d_{\a\b}) \dm_0(\a,\b).
\]
With the same justification in case of jumps, we calculate the derivative of $\cG$:
\[
\begin{split}
\ddt \cG  & = -  \int_{\T^2} |u_{\a\b}|^2 \psi'(d_{\a\b}) \dm_0(\a,\b)+  \int_{\T^3} \sign(u_{\a\b}) \psi(d_{\a\b}) \phi_{\a\g} u_{\g\a} \dm_0(\a,\b,\g) \\
& \leq a\cI_2 - b \cV_2 + \cR.
\end{split}
\]
Here,
\[
\begin{split}
\cR & =  \int_{\T^3} \sign(u_{\a\b}) \psi(d_{\a\b}) \phi_{\a\g} u_{\g\a} \dm_0(\a,\b,\g) \\
& = \frac12  \int_{\T^3} [\sign(u_{\a\b}) \psi(d_{\a\b}) - \sign(u_{\g\b}) \psi(d_{\g\b})] \phi_{\a\g} u_{\g\a} \dm_0(\a,\b,\g) \\
& \leq  \int_{\T^3} (\psi(d_{\a\b}) -  \psi(d_{\g\b})) \phi_{\a\g} u_{\g\a} \one_{u(\b) < \min\{u(\a),u(\g)\}}  \dm_0(\a,\b,\g)\\
&+ \int_{\T^3} ( \psi(d_{\g\b}) - \psi(d_{\a\b}) ) \phi_{\a\g} u_{\g\a} \one_{u(\b) > \max\{u(\a),u(\g)\} } \dm_0(\a,\b,\g)\\
& \leq  \int_{\T^2} |x_{\a\g}| | u_{\g\a} |\phi_{\a\g} \dm_0(\a,\g).
\end{split}
\]
In case (ii) we then obtain
\[
\cR \leq \frac{c}{t} + b t \cI_2,
\]
and the proof ends as in the discrete settings. In case (iii) we consider the collision potential
\[
\cC = \int_{\T^2} \frac{   \dm_0(\a,\b) }{(|x_{\a\b}|\wedge r_0)^{\b-2}} .
\]
Note that it is well-posed for $\b<3$ (in view also of the fact that the density is bounded for regular solutions).  Computation similar to the discrete case proves \eqref{e:Cbound}, and from this point on the proof is exactly the same.
\end{proof}

Next, we prove the $L^\infty$-based result,  \thm{t:1DTcont}. Heuristically, the alignment still holds in spite of all the lacking  assumptions because in regions where the density is non-negligible the alignment term works faster than the transport to avoid agent  collisions. At the same time in regions where the density is thin, the equation acts as Burgers. So, to avoid a blowup it must have low velocity variations in those regions as well.  Let us note that technically the theorem works even if $\phi = 0$, although in this case it is trivial -- the only global smooth solution to Burgers on $\T$ is a constant one. Positivity of $\phi$ makes it possible to have global smooth solutions as discussed in \cite{CCTT2016}.

\begin{proof}[Proof of \thm{t:1DTcont}] By the Galilean invariance we can assume throughout that $\bar{v}=0$.	As a consequence of the energy equality and \eqref{e:V2decay} we obtain
	\[
	\int_T^\infty \int_{\T^2} \phi_{\a \b} |u(\a,t) - u(\b,t) |^2 \dm_0(\a,\b) \dt \leq  C \frac{\ln T}{T}: = \e.
	\]
	Denote
	\[
	F(\a,T) = \int_{T}^\infty  \int_{\T} \phi_{\a \b} |u(\a,t) - u(\b,t) |^2 \dm_0(\b) \dt.
	\]
	So, we have
	\[
\int_\T F(\a,T) \dm_0(\a) \leq  \e.
\]
Let us fix another small parameter $\d>0$ and define the ``good set":
	\[
	G_\delta(T) = \{\alpha : F(\alpha,T) \le \delta \}.
	\]
	We denote by $G^c_\delta$ the complement of $G_\d$, so that $m_0(G^c_\delta ) = M -m_0(G_\delta)$ (recall that $M$ is the total mass of the flock). By the Chebychev inequality,
	\begin{equation}\label{e:CF}
	m_{0}(G^c_\delta) <\frac{\e}{\d}.
	\end{equation}
Thus, the good set occupies almost all of the domain provided $\e \ll \d$.   We now proceed by proving that alignment occurs first on the good set identified above, and then on the rest of the torus later in time within a controlled time scale.

\begin{lemma}[Alignment on the good set]\label{l:onG}
	We have
	\begin{equation*}
	\sup_{\a_1,\a_2 \in G_\d(T),\, t \geq T} |u(\alpha_1,t) - u(\alpha_2,t)| \lesssim \delta^{2/3}.
	\end{equation*}
\end{lemma}
\begin{proof} Note that it suffices to establish alignment at time $T$ only. This simply follows from monotonicity of the our $F$-function:
	\[
	F(\a,t) \leq F(\a,T), \quad t>T,
	\]
which implies that the good sets are increasing in time, $G_\d(T) \ss
G_\d(t)$.

 Integrating the equation
	\[
	\ddt u(\a,t) =  \int_{\T} \phi_{\a\b} u_{\a\b} \dm_0(\b)
	\]
over $[T,t]$ for any  $\a \in G_\d$ we obtain
	\begin{equation}\label{e:tT}
	| u(\a,t) - u(\a,T) |  \leq \int_{T}^t \int_{\T} \phi_{\a\b}| u_{\a\b}| \dm_0(\b)  \lesssim \d \sqrt{t-T}.
	\end{equation}
Suppose that for some \(\alpha_1,\alpha_2 \in G_{\delta}\) we have
\[
u(\alpha_1,T) - u(\alpha_2,T) > U,
\]
where $U$ to be determined later. Then 	in view of \eqref{e:tT},
\[
u(\alpha_1,t) - u(\alpha_2,t) > \frac{U}{2},
\]
as long as
	\begin{equation}\label{e:tTabove}
	t- T \lesssim \frac{U^2}{\d^2}.
	\end{equation}
During this time the corresponding characteristics will undergo  a significant relative displacement
	\[
	x(\a_1,t) - x(\a_2,t) \geq x(\a_1,T) - x(\a_2,T) + \frac12 U(t- T) \mod 2\pi,
	\]
where  $\frac12 U(t- T)> 4\pi$ as long as $t- T \gtrsim \frac{1}{U}$. If this is allowed to happen, then the characteristics will find themselves at the separation  distance equal $2 \pi = 0$, so they collapse. We necessarily obtain
	\begin{equation}\label{e:Ut}
\frac1U \gtrsim \frac{U^2}{\d^2},
	\end{equation}
	which gives $U \lesssim \d^{2/3}$ as claimed.
\end{proof}

Next step is to show that the solution aligns completely at a not too distant later time $t>T$.

	\begin{lemma}[Alignment outside the good set]\label{l:outG}
	For all $t\gtrsim T + \frac{1}{\d^{1/3} + (\e/\d)^{1/2}}$ we have
		\begin{equation*}
	\sup_{\a\in \T,\g \in G_\d(T)} |u(\alpha,t) - u(\g,t)| \lesssim \d^{1/3} + (\e/\d)^{1/2}.
	\end{equation*}
\end{lemma}
\begin{proof}
Let $\a \in \T$ and $\g \in G_\d(T)$. Let us write the momentum equation in Lagrangian coordinates as follows
\[
	\begin{split}
	\ddt u(\a,t) &=\int_\T u_{\b\a} \phi_{\a\b} \dm_0(\b) = \int_\T (u_{\b\g} + u_{\g\a}) \phi_{\a\b} \dm_0(\b)  \\
	&=   (\phi \ast \rho)(x(\a,t),t) u_{\g\a} + \int_\T u_{\b\g} \phi_{\a\b} \dm_0(\b).
	\end{split}
\]
The integral term will remain small for all $t\geq  T$, in view \lem{l:onG} and \eqref{e:CF}. Indeed,
	\[
	\begin{split}
	\left| \int_\T u_{\b\g}(t) \phi_{\a\b} \dm_0(\b) \right| & = \left|   \int_{G_\d(T) } u_{\b\g}(t) \phi_{\a\b} \dm_0(\b) \right| + \left|   \int_{G^c_\d(T)} u_{\b\g}(t)  \phi_{\a\b} \dm_0(\b) \right|\\
	& \lesssim \d^{2/3}  + \frac{\e}{\d}.
	\end{split}
	\]
	So,
	\begin{equation}\label{e:ddvboth}
	  (\phi \ast \rho) u_{\g\a} -  \d^{2/3} - \frac{\e}{\d} \leq \ddt u (\a,t) \leq   (\phi \ast \rho) u_{\g\a} + \d^{2/3}  + \frac{\e}{\d}.
	\end{equation}
Let us fix a time $t\gtrsim T + \frac{1}{\d^{1/3} + (\e/\d)^{1/2}}$, and assume that $u_{\a\g}(t) = U >0$, where $U$ is to be determined later. Let us drive the dynamics backwards in time from the moment $t$.  For a time period $[s,t]$, where $T<s<t$, the difference will remain positive $u_{\a\g}(s) >0$. On that time period, the right hand side of \eqref{e:ddvboth} implies
	\[
	\ddt u \leq  \d^{2/3}  + \frac{\e}{\d}
	\]
	and hence,
	\[
	 u(\a,t) - \left( \d^{2/3}  + \frac{\e}{\d}\right)(t-s) \leq u(\a,s).
	\]
	At the same time, by \eqref{e:tT} applied to $\g \in G_\d$, we have
	\[
	|u(\g,t) - u(\g,s)| \leq \d(t-s)^{1/2}.
	\]
So, 	combined with the previous,
\[
U - \left( \d^{2/3}  + \frac{\e}{\d}\right)(t-s) - \d(t-s)^{1/2} =  u_{\a\g}(t) - \left( \d^{2/3}  + \frac{\e}{\d}\right)(t-s) - \d(t-s)^{1/2} \leq u_{\a\g}(s).
 \]
We will have
\[
u_{\a\g}(s) \geq \frac{U}{2},
\]
as long as  $(t-s) \lesssim \frac{U}{ \d^{2/3}  + \frac{\e}{\d}}$ and $(t-s) \lesssim \frac{U^2}{\d^2}$.  The former is more restrictive, unless $U\lesssim \d^{4/3}$, in which case we have acheived our objective.  Arguing as in the proof of \lem{l:onG} we obtain collision backwards in time, provided $(t-s) \sim  1/U$.  This is possible when $U \gtrsim \d^{1/3} + (\e/\d)^{1/2}$ on the time interval $t- T \gtrsim 1/U$, which is true under the assumption.

Arguing from the opposite end, $u_{\a\g}(t) = - U <0$, we obtain the bound from below as well.
\end{proof}

\lem{l:outG} implies the global alignment: 	for all $t\gtrsim T + \frac{1}{\d^{1/3} + (\e/\d)^{1/2}}$
\begin{equation*}
\sup_{\a,\g\in \T} |u(\alpha,t) - u(\g,t)| \lesssim \d^{1/3} + (\e/\d)^{1/2}.
\end{equation*}
Optimizing over $\d$, we pick $\d = \e^{3/5}$, and recalling that $\e = \ln T/ T$, we obtain
\begin{equation*}
\sup_{\a,\g\in \T} |u(\alpha,t) - u(\g,t)| \lesssim	\left(\frac{\ln T}{T}\right)^{1/5},
\end{equation*}
for $t \sim T + \left(\frac{T}{\ln T}\right)^{1/5} \sim T$.  This proves the result.

\end{proof}


\end{document}